\documentclass[12pt]{article}

\usepackage{amsfonts}
\usepackage [dvips]{graphics}
\usepackage[centertags]{amsmath}
\usepackage{amsfonts}
\usepackage{amssymb}
\usepackage{amsthm}
\usepackage{newlfont}
\usepackage{mathrsfs}
\usepackage [dvips]{graphics}
\usepackage[centertags]{amsmath}
\usepackage{amsfonts}
\usepackage{amssymb}
\usepackage{amsthm}
\usepackage{newlfont}
\usepackage{hyperref}

\newlength{\defbaselineskip}
\newcommand{\setlinespacing}[1]%
           {\setlength{\baselineskip}{#1 \defbaselineskip}}

\theoremstyle{plain}

\newtheorem{thm}{Theorem}[section]

\newtheorem{pro}[thm]{Problem }

\theoremstyle{definition}

\newtheorem{ass}{Assumption}[section]

\makeatletter\@addtoreset{equation}{section} \makeatother

\begin{document}

\title{ Non-zero Sum Stochastic Differential Games of Fully Coupled
Forward-Backward Stochastic
 Systems }
Qingxin Meng$^1$\and Yongzheng Sun$^2$}  
\date{}

\footnotetext[1]{Department of Mathematics, Huzhou University, Huzhou,  313000,China, Email: 071018034@fudan.edu.cn.}
\footnotetext[2]{School of Sciences, China University of Mining and Technology, Xuzhou, 221008, China, Email: yzsung@gmail.com}

\maketitle

\begin{abstract}In this paper,  an open-loop two-person
non-zero sum stochastic
differential game is considered for forward-backward stochastic
 systems. More
precisely, the controlled systems are described by a fully coupled
nonlinear multi- dimensional forward-backward stochastic
differential equation driven by  a multi-dimensional Brownian
motion.   one sufficient (a verification theorem) and one necessary conditions
 for the existence of  open-loop Nash equilibrium  points for the corresponding two-person
non-zero sum stochastic differential game are proved. The control
domain need to be convex and the admissible controls for both
players are allowed to appear in both the drift and diffusion of the
state equations.
\end{abstract}


\maketitle

\section{ Introduction }

Differential Game theory had been an active area of research and a useful tool in
many applications, particularly in biology and economic. The so called differential
games  are the ones in which the position, being controlled by the players,
 evolves continuously. On the one hand, since the study of differential games was initiated by Isaacs \cite{Isa54},
 many papers (see \cite{Berk6401,Berk6402,Berk6701,Berk6702,Berk85,ElKa72,Flem61,Flem64}) have appeared which developed the foundations for two-person zero sum  differential
 games.  For this case, there  a single performance criterion which one player tries to minimize and the
 other tries to maximize.  On  the  other hand, many authors (see
 \cite{Frie69,Case69,ChZh,Eise82,LuRu71,SaPa71,SaRaPr,Sta69,
 StHo6901,StHo6902,Var70}discussed N-person non-zero sum differential
 games. For this case, there may be more than two players and each player tries to
 minimize his individual performance criterion, and the sum of all player's criteria
 is not zero or is it constant.

  All the above mentioned paper are restricted deterministic system. On
  the differential games of  stochastic systems, we can refer to\cite{BaLi,Frie72,Uch78}.
  In 2008, Tang and Li \cite{TaLi} established  the minimax  principle for  N-person differential games governed by
  forward stochastic systems with the control appearing in the diffusion term.
  In 2010, wang and Yu \cite{WaYu} studied the Non-zero sum differential
  games of backward stochastic systems, and they established
  a necessary condition and a sufficient condition in the form of
  stochastic maximum principle for open-loop
  Nash equilibrium.

  Forward-Backward stochastic systems   are not only used in mathematical economics (see Antonelli
\cite{Anto}, Duffie and Epstein \cite{DuEp}, for example), but also used in
mathematical finance(see El Karoui, Peng and Quenez \cite{ElPe}). It now
becomes more clear that certain important problems in mathematical
economics and mathematical finance, especially in the optimization
problem, can be formulated to be Forward-backward stochastic
system.  So the optimal control problem for Forward-backward stochastic
system and the corresponding stochastic maximum principle are extensively
studied in this literature. We refer to \cite{Xu95,Wu9, ShWz2006} and references therein. They
established the necessary maximum principle in the case the control
domain is convex or the forward diffusion coefficients can not
contain a control variable.  In 2010, Yong\cite{Yong10} proved necessary conditions for
the optimal control of forward-backward stochastic systems where the control
domain is not assumed to be convex and the control appears
in the diffusion coefficient of the
forward equation.

In this paper we will discuss non-zero sum stochastic
  differential games for forward-backward  stochastic systems.
  To our best knowledge, very little work has been published
on this subject.  In section 2, we state  the
problem and our main assumptions. In section 3, we state and prove
our main results: a sufficient condition for the existence
of open-loop Nash equilibrium point which can check
whether the candidate equilibrium points are optimal or not. Section 4
is devoted to present a necessary condition for
the existence of open-loop
Nash equilibrium point  by  the stochastic maximum principle
for the optimal control of the optimal
control problem of forward-backward stochastic systems established in \cite{Wu9}.

Moreover, we refer to \cite{Peng99,MaYo99} on the existence and
uniqueness of solutions to the fully coupled forward-backward
stochastic differential equations.

 \section{Problem formulation and main assumptions}

Let $(\Omega,{\cal F}, \{{\cal F}_t\}_{t\geq 0}, P)$ be a
complete probability space, on which a $d$-dimensional
standard Brownian motion $B(\cdot)$ is defined with
$\{{\cal F}_t\}_{t\geq 0}$  being its natural filtration, augmented
by all $P$-null sets in ${\cal F}.$  Let $T>0$
be a fixed time horizon.
Let E be a Euclidean space. The inner product in E is denoted by $\langle \cdot, \cdot\rangle $,  and the norm in E
is denoted  by $|\cdot|.$
 We further introduce some other spaces that will be used in the paper.
 Denote by $L^2(\Omega, {\cal F}_T, P; E)$  the the set of all
 $E$-valued ${\cal F}_T$-measurable random variable $\eta$ such that
 $E|\eta|^2<\infty.$
  Denote by $M^2(0, T; E)$ the
set of all $E$-valued  $\cal F_t$-adapted stochastic processes
$\{\varphi(t): t\in [0, T] \}$ which satisfy $E\int_0^T|\varphi(t)|^2dt<\infty.$
 Denote by $\cal S^2(0, T; E)$ the
set of all $E$-valued   $\cal F_t$-adapted continuous stochastic processes
$\{\varphi(t): t\in [0, T] \}$ which satisfy $E\sup_{0\leq t\leq
 T}|\varphi(t)|^2dt<\infty.$

In this paper, we consider
    the system which is given by a
   controlled fully coupled nonlinear forward-backward stochastic differential equations (abbr. FBSDEs)
   of the form
   \begin{equation}\label{eq:1.1}
\displaystyle\left\{\begin{array}{lll}
dx(t)&=&b(t,x(t),y(t),z(t),u_1(t),u_2(t))dt\\&~~&+\sigma(t,x(t),y(t),z(t),u_1(t),u_2(t))dB(t),\\\displaystyle
dy(t)&=&-f(t,x(t),y(t),z(t),u_1(t),u_2(t))dt\\&~~&+z(t)dB(t),\\\displaystyle
x(0)&=&a,\\\displaystyle y(T)&=&\xi.

\end{array}
\right.
\end {equation}
Here $
\displaystyle b: [0,T]\times R^n\times R^m\times R^{m\times
d}\times { U}_1\times{
U}_2\rightarrow R^n,
\displaystyle \sigma: [0,T]\times R^n\times R^m\times R^{m\times
d}\times { U}_1\times{
U}_2\rightarrow R^{n\times d},
\displaystyle f:[0,T]\times R^n\times R^m\times R^{m\times d}\times
{ U}_1\times{
U}_2\rightarrow R^m
$
are given mapping,  $a$
and $T>0$ are given constants, and$\xi\in L^2(\Omega, {\cal F}_T, P; R^m)$. The processes $u_1(\cdot)$ and
$u_2(\cdot)$ in the system \eqref{eq:1.1} are the  open-loop
control processes which present the controls of the two players,
required to have values in two given nonempty convex sets ${
U}_1\subset R^{k_1}$ and ${ U}_2\subset R^{k_2}$ respectively.
The  admissible control process $(u_1(\cdot), u_2(\cdot))$ is defined as a ${\cal
F}_t$-adapted process with values in $U_1\times U_2$ such that
$E\displaystyle\int_0^T(|u_1(t)|^2+|u_2(t)|^2)dt<+\infty.$  The set of all
admissible control processes  is denoted by ${\cal A}_1\times{\cal A}_2.$

For each one of the two player,  there is a cost functional
\begin{equation}
\begin{array}{ll}
&J_i(u_1(\cdot),
u_2(\cdot))\\=&E\bigg[\displaystyle\int_0^Tl_i(t,x(t),y(t),z(t),u_1(t),
u_2(t))dt\\&+\phi_i(x(T))+h_i(y(0))\bigg],
\end{array}
\end{equation}
 where
$ l_i: [0,T]\times R^n\times R^m\times R^{m\times
d}\times {\cal
U}_1\times{\cal U}_2\rightarrow R,
\displaystyle\phi_i: R^n \rightarrow R,
\displaystyle h_i: R^m \rightarrow R
$
 are given mapping $(i=1,2)$.

 Now we make  the main assumptions throughout the paper.

 \begin{ass}\label{ass:2.1}
 $f,g, \sigma$   are  continuously  differentiable  with  respect  to  $(x,y,z,u_1, u_2)$.
The  derivatives  of  $f, g, \sigma$ are bounded.
For any admissible control $(u_1(\cdot),u_2(\cdot)),$
the forward-backward stochastic system satisfies  the assumptions
(H2.1) and (H2.2) in Wu\cite{Wu9}.
\end{ass}
\begin{ass}\label{ass:2.2}
$l_i,\phi_i$ and $h_i$  are  continuously  differentiable  with  respect  to  $(x,y,z,u_1, u_2), x$
and $y, (i=1,2).$
And $l_i$ is bounded by $C(1+|x|^2+|y|^2+|z|^2+|u_1|^2+
|u_2|^2).$ And the  derivatives of  $l_i$ are bounded by $C(1+|x|+|y|+|z|+|u_1|+|u_2|).$
And $\phi_i$ and $h_i$ are bounded by $C(1+|x|^2)$ and $C(1+|y|^2)$ respectively.
And the  derivatives  of $\phi_i$ and $h_i$ with respect to $x$ and $y$
are bounded by $C(1+|x|)$ and $C(1+|y|)$ respectively. $(i=1,2)$.
\end{ass}
 Under Assumption \ref{ass:2.1},  from Theorem 2.1 in Wu \cite{Wu9},  we see that
 for any given admissible control  $(u_1(\cdot), u_2(\cdot)$, the system \eqref{eq:1.1}
 admits a unique solution $$(x(\cdot), y(\cdot), z(\cdot))\in S_{\cal
F}^2(0,T;R^n)\times\in S_{\cal
F}^2(0,T;R^m)\times \in M_{\cal
F}^2(0,T; R^{m\times d}).$$  Then we call $(x(\cdot), y(\cdot), z(\cdot))$ the
state process corresponding to the control process $(u_1(\cdot), u_2(\cdot)$ and
 $((u_1(\cdot), u_2(\cdot); y(\cdot), q(\cdot), z(\cdot))$ the
admissible pair. Furthermore, from Assumption \ref{ass:2.2},
it is easy to check that$ |J_i(u_1(\cdot), u_2(\cdot))|<\infty.$$(i=1,2).$

Then we can pose the following two-person non-zero
sum stochastic differential game problem

\begin{pro} \label{pro:2.1}
  Find an open-loop admissible control $(\bar{u}_1(\cdot), \bar u_2(\cdot))
  \in \cal A_1\times \cal A_2$  such that
 \begin{equation} \label{eq:b7}
J_1(\bar{u}_1(\cdot), \bar u_2(\cdot))=\displaystyle\inf_{u_1(\cdot)\in {\cal
A}_1}J_1({u}_1(\cdot), \bar u_2(\cdot))
\end{equation}
and
\begin{equation} \label{eq:b7}
J_2(\bar{u}_1(\cdot), \bar u_2(\cdot))=\displaystyle\inf_{u_2(\cdot)\in {\cal
A}_2}J_2(\bar{u}_1(\cdot),  u_2(\cdot)).
\end{equation}

\end{pro}
 Any    $(\bar{u}_1(\cdot), \bar u_2(\cdot))
  \in \cal A_1\times \cal A_2$   satisfying the above is called
  a open-loop Nash equilibrium point of Problem \ref{pro:2.1}. Such an admissible control
  allows two players to
  play individual
  optimal control strategies simultaneously.

\section{A Verification Theorem}

In this section we state and prove  a verification theorem
for the Nash equilibrium points of  Problem \ref{pro:2.1}.

 For any given admissible pair $(u_1(\cdot), u_2(\cdot);  x(\cdot), y(\cdot), z(\cdot)),$
  We can  introduce the following adjoint forward-backward
   stochastic differential  equations of the system \eqref{eq:1.1}
  \begin {equation}
  \left\{\begin{array}{ll}\label{eq:2.1}
      \displaystyle dk^i(t)=&-\bigg[b_y^*(t,x(t),y(t),z(t),u_1(t),u_2(t))p^i(t)\\
 &+\sigma_y^*(t,x(t),y(t),z(t),u_1(t),u_2(t))q^i(t)
 \\&-f_y^*(t,x(t),y(t),z(t),u_1(t),u_2(t))k^i(t)\\
 &
 \displaystyle+l_{iy}(t,x(t),y(t),z(t),u_1(t),u_2(t))\bigg]dt\\
 ~~&-\bigg[b_z^*(t,x(t),y(t),z(t),u_1(t),u_2(t))p^i(t)\\
 &
\displaystyle
+\sigma_z^*(t,x(t),y(t),z(t),u_1(t),u_2(t))q^i(t)
\\~~&-f_z^*(t,x(t),y(t),z(t),u_1(t),u_2(t))k^i(t)\\
& +l_{iz}(t,x(t),y(t),z(t),u_1(t),u_2(t))\bigg]dB(t)\\\displaystyle
dp^i(t)=&-\bigg[b_x^*(t,x(t),y(t),z(t),u_1(t),u_2(t))p^i(t)\\
&+\sigma_x^*(t,x(t),y(t),z(t),u_1(t),u_2(t))q^i(t)\\~~&
-f_x(t,x(t),y(t),z(t),u_1(t),u_2(t))k^i(t)\\
&\displaystyle
+l_{ix}(t,x(t),y(t),z(t),u_1(t),u_2(t))\bigg]dt\\
&+q^i(t)dB(t)
\\\displaystyle k^i(0)=&-h_{iy}(y_0),
~~~p^i(T)=\phi_{ix}(x(T)),\\&~~0\leq t\leq T, (i=1.2).
\end{array}
\right.
\end {equation}
Under Assumptions\ref{ass:2.1}-\ref{ass:2.2}, according to Theorem 2.2 in \cite{Wu9},
 the above adjoint equation has a unique solution
$(k^i(\cdot), p^i(\cdot), q^i(\cdot))\in {\cal S}_{\cal F}(0,T; R^m)\times \in {\cal S}_{\cal F}^2(0,T; R^n)
\times \in M_{\cal F}(0,T; R^{n\times d}), (i=1.2).$

  We define the Hamiltonian functions $H_i:[0,T]\times R^n\times
R^m\times R^{m\times d}\times {\cal U}_1\times {\cal U}_2\times
R^n\times R^{n\times d}\times R^m\rightarrow R$ by
\begin {equation}\label{eq:2.2}
\begin{array}{ll}
&\displaystyle H_i(t,x,y,z,u_1,u_2,p,q,k)=\langle k,
-f(t,x,y,z,u_1,u_2\rangle\\
&+\langle p, b(t,x,y,z,u_1,u_2)\rangle +l_i(t,x,y,z,u_1,u_2)\\\displaystyle
&+\langle q,\sigma(t,x,y,z,u_1,u_2)\rangle,  (i=1,2).
\end{array}
\end {equation} Then we can rewrite the equations \eqref{eq:2.1}  in Hamiltonian
system's form:
\begin {equation}
  \left\{\begin{array}{ll}
  \displaystyle dk^i(t)&=-H_{iy}(t,x(t),y(t),z(t),u_1(t),u_2(t),p^i(t),q^i(t),k^i(t))dt\\
  &-H_{iz}(t,x(t),y(t),z(t),u_1(t),u_2(t),p^i(t),q^i(t),k^i(t))dB(t)\\\displaystyle
dp^i(t)&=-H_{ix}(t,x(t),y(t),z(t),u_1(t),u_2(t),p^i(t),q^i(t),k^i(t))dt\\
&+q^i(t)dB(t)\\
 k^i(0)=&-h_{iy}(y_0),
~~~p^i(T)=\phi_{ix}(x(T)), (i=1,2).
  \end{array}
  \right.
  \end {equation}

We are now coming to a verification theorem for an Nash equilibrium point of
 Problem \ref{pro:2.1}.
\begin{thm}

Under Assumptions \ref{ass:2.1}-\ref{ass:2.2}, let $(\bar{u}_1(\cdot), \bar{u}_2(\cdot);
 \bar{x}(\cdot), \bar{y}(\cdot), \bar{z}(\cdot))$
  be an admissible pair.
 Let $({\bar p^i}(\cdot), \bar q^i(\cdot), \bar k^i(\cdot) )$$(i=1,2)$ be  the unique
 solution of the corresponding adjoint equation \eqref{eq:2.1}.  Suppose that
 for almost all  $(t,\omega)\in [0,T]\times \Omega$ ,
 $(x,y,z,u_1)\mapsto H_1(t,x,y,z,{u}_1,\bar{u}_2(t),\\ \bar{p}^1(t),\bar{q}^1(t),\bar{k}^1(t))$
 is convex with respect to  $(x,y,z,u_1)$,
 $(x,y,z,u_2)\mapsto H_2(t,x,y,z,\bar{u}_1(t),{u}_2,\bar{p}^2(t),\\\bar{q}^2(t),\bar{k}^2(t))$
 is convex with respect to  $(x,y,z,u_2)$, $x\mapsto h_i(x)$ is convex with respect with to
   $x$, and $y\mapsto \phi_i(y)$ is convex with respect to $y$ (i=1,2), and the
 following optimality condition holds
 \begin {equation}\label{eq:2.4}
\begin{array}{ll}
&\displaystyle\max_{u_1\in{\cal U}_1}
H_1(t,\bar{x}(t),\bar{y}(t),\bar{z}(t),u_1,\bar{u}_2(t),\bar{p}^1(t),\bar{q}^1(t),\bar{k}^1(t))
\\&
~~~~=H_1(t,\bar{x}(t),\bar{y}(t),\bar{z}(t),\bar{u}_1(t),\bar{u}_2(t),\bar{p}^1(t),\bar{q}^1(t),\bar{k}^1(t)),
\end{array}
\end {equation}
and
\begin {equation}\label{eq:2.5}
\begin{array}{ll}
&\displaystyle\max_{u_2\in {\cal
U}_2}H_2(t,\bar{x}(t),\bar{y}(t),\bar{z}(t),\bar{u}_1(t),u_2,\bar{p}^2(t),\bar{q}^2(t),\bar{k}^2(t))\\&
~~~~=H_2(t,\bar{x}(t),\bar{y}(t),\bar{z}(t),\bar{u}_1(t),\bar{u}_2(t),\bar{p}^2(t),\bar{q}^2(t),\bar{k}^2(t)).
\end{array}
\end {equation}
 Then $(\bar{u}_1(\cdot), \bar{u}_2(\cdot))$ is Nash equilibrium point of
 Problem \ref{pro:2.1}

\end{thm}

\begin{proof} (i) we  consider an
stochastic optimal control problem. The system  is the following
controlled
forward-backward stochastic differential equation
\begin{equation}\label{eq:2.5}
\displaystyle\left\{\begin{array}{lll}
dx(t)&=&b(t,x(t),y(t),z(t),u_1(t),\bar{u}_2(t))dt\\
&&+\sigma(t,x(t),y(t),z(t),u_1(t),\bar{u}_2(t))dB(t),\\\displaystyle
dy(t)&=&-f(t,x(t),y(t),z(t),u_1(t),\bar{u}_2(t))dt\\
&&+z(t)dB(t),\\\displaystyle x(0)&=&a\\\displaystyle y(T)&=&\xi,
\end{array} \right.
\end {equation}

where $u_1(\cdot)$ is  any given admissible control in $\cal A_1.$
The cost function is defined as
\begin{equation}\label{eq:3.7}
\begin{array}{ll}
&J_1(u_1(\cdot),\bar{u}_2(\cdot))\\
=&E\bigg[\displaystyle\int_0^Tl_1(t,x(t),y(t),z(t),u_1(t),
\bar{u}_2(t))dt\\
&+\phi_1(x(T))+h_1(y_0)\bigg], \end{array}
\end{equation}
where $(x(\cdot), y(\cdot),z(\cdot))$ is
the solution  to the forward-backward stochastic system \eqref{eq:2.5} corresponding
to the control $u_1(\cdot)\in \cal A_1.$

 The  optimal control problem is minimize
$J(u_1(\cdot), \bar{u}_2(\cdot))$ over $u_1(\cdot) \in {\cal A }_1$.
Now will show  the admissible control $\bar{u}_1(\cdot)$ is  an optimal
control of the problem, i.e,
 \begin{equation} J_1(\bar{u}_1(\cdot),
\bar{u}_2(\cdot))=\displaystyle\min_{u_1(\cdot)\in {\cal
A}_1}J_1(u_1(\cdot), \bar{u}_2(\cdot)).
\end{equation}

In fact, Let $u_1(\cdot)$ be  any admissible control in ${\cal A}_1,$
$(x(\cdot),y(\cdot),z(\cdot))$ be the corresponding state process of the
system \eqref{eq:2.5}. It is easy to check that for the control  $\bar u_1(\cdot)$,
the corresponding state process of the system \eqref{eq:2.5} is
indeed
$(\bar{x}(\cdot), \bar y(\cdot), \bar z(\cdot)).$

From \eqref{eq:3.7}, we have
\begin{equation}\label{eq:2.9}
\begin{array}{ll}
 &J_1(u_1(\cdot),\bar u_2(\cdot))-J_1(\bar{u}_1(\cdot), \bar u_2(\cdot))\\&~~~~~~=E\displaystyle\int_0^T\bigg[l_1(t,x(t),y(t),z(t),u_1(t), \bar u_2(t))\\&~~~~~~~~~-l_1(t,\bar{x}(t),\bar{y}(t),\bar{z}(t),\bar{u}_1(t), \bar u_2(t))\bigg]dt
\\&~~~~~~~~~+E\bigg[\phi_1(x(T))-\phi_1(\bar{x}(T))\displaystyle\bigg]
\\&~~~~~~~~~+E\bigg[h_1(y(0))-h_1(\bar{y}(0))\bigg]\\&~~~~~~=I_1+I_2,
\end{array}
\end{equation}
where \begin{eqnarray}
  \begin{split}
    I_1&=E\displaystyle\int_0^T\bigg[l_1(t,x(t),y(t),z(t),u_1(t), \bar u_2(t))\\&~~~~-l_1(t,\bar{x}(t),\bar{y}(t),\bar{z}(t),\bar{u}_1(t), \bar u_2(t))\bigg]dt,
  \end{split}
\end{eqnarray}

\begin{equation}
\begin{array}{ll}
\displaystyle I_2=E\bigg[\phi_1(x(T))-\phi_1(\bar{x}(t))\displaystyle\bigg]
+E\bigg[h_1(y(0))-h_1(\bar{y}(0))\bigg].
\end{array}
\end{equation}

Using Convexity of $\phi_1$ and $h_1$,  and It\^{o} formula to
$\langle \bar{p}^1(t),x(t)-\bar{x}(t)\rangle +\langle \bar{k}^1(t),y(t)-\bar{y}(t)\rangle,$
 we get

\begin{equation}\label{eq:2.12}
\begin{array}{ll}
I_2&=E\big[\phi_1(x(T))-\phi_1(\bar{x}(T))\big]+E\big[h_1(y(0)-h_1(\bar{y}(0))\big]\\
&\geq
E\langle \phi_{1x}(\bar{x}(T)),x(T)-\bar{x}(T)\rangle +E\langle h_{1y}(\bar{y}_0),y_0-\bar{y}_0\rangle
\\&=E\langle \bar p^1(T)),x(T)-\bar{x}(T)\rangle +E\langle  \bar k^1(0),y_0-\bar{y}_0\rangle
\\
&=-E\displaystyle\int_0^T\langle H_{1x}(t,\bar{x}(t),\bar{y}(t),\bar{z}(t),\bar u_1(t),\bar u_2(t),\bar{p}^1(t),
\bar{q}^1(t),
\bar{k}^1(t)),x(t)-\bar{x}(t)\rangle dt\\
&~~-E\displaystyle\int_0^T\langle H_{1y}(t,\bar{x}(t),\bar{y}(t),\bar{z}(t),\bar u_1(t),\bar u_2(t),\bar{p}^1(t),
\bar{q}^1(t),
\bar{k}^1(t)),y(t)-\bar{y}(t)\rangle dt\\
&~~-E\displaystyle\int_0^T\langle H_{1z}(t,\bar{x}(t),\bar{y}(t),\bar{z}(t),\bar u_1(t),\bar u_2(t),\bar{p}^1(t),
\bar{q}^1(t), \bar{k}^1(t)),z(t)-\bar{z}(t)\rangle dt
\\
&~~+E\displaystyle\int_0^T\langle\bar{p}^1(t),b(t,x(t),y(t),z(t),u_1(t),\bar u_2(t))-b(t,\bar{x}(t),\bar{y}(t),\bar{z}(t),\bar u_1(t),\bar u_2(t))
\rangle dt\\
&~~+E\displaystyle\int_0^T\langle\bar{q}^1(t),\sigma(t,x(t),y(t),z(t),u_1(t), \bar u_2(t))
-\sigma(t,\bar{x}(t),\bar{y}(t),\bar{z}(t),\bar u_1(t),\bar u_2(t))\rangle dt
\\&~~+E\displaystyle\int_0^T\langle \bar{k}^1(t),-(f(t,x(t),y(t),z(t),u_1(t), \bar u_2(t))-f(t,\bar{x}(t),\bar{y}(t),\bar{z}(t),\bar{u}_1(t), \bar{u}_2(t)))\rangle dt\\
&=-J_1+J_2,
\end{array}
\end{equation}
where
$$
\begin{array}{ll}
J_1&=E\displaystyle\int_0^T\langle H_{1x}(t,\bar{x}(t),\bar{y}(t),\bar{z}(t),\bar u_1(t),\bar u_2(t),\bar{p}^1(t),
\bar{q}^1(t),
\bar{k}^1(t)),x(t)-\bar{x}(t)\rangle dt\\
&~~+E\displaystyle\int_0^T\langle H_{1y}(t,\bar{x}(t),\bar{y}(t),\bar{z}(t),\bar u_1(t),\bar u_2(t),\bar{p}^1(t),
\bar{q}^1(t),
\bar{k}^1(t)),y(t)-\bar{y}(t)\rangle dt\\
&~~+E\displaystyle\int_0^T\langle H_{1z}(t,\bar{x}(t),\bar{y}(t),\bar{z}(t),\bar u_1(t),\bar u_2(t),\bar{p}^1(t),
\bar{q}^1(t), \bar{k}^1(t)),z(t)-\bar{z}(t)\rangle dt,\\
J_2=&E\displaystyle\int_0^T\langle\bar{p}^1(t),b(t,x(t),y(t),z(t),u_1(t),\bar u_2(t))-b(t,\bar{x}(t),\bar{y}(t),\bar{z}(t),\bar u_1(t),\bar u_2(t))
\rangle dt\\
&~~+E\displaystyle\int_0^T\langle\bar{q}^1(t),\sigma(t,x(t),y(t),z(t),u_1(t), \bar u_2(t))
-\sigma(t,\bar{x}(t),\bar{y}(t),\bar{z}(t),\bar u_1(t),\bar u_2(t))\rangle dt
\\&~~+E\displaystyle\int_0^T\langle \bar{k}^1(t),-(f(t,x(t),y(t),z(t),u_1(t), \bar u_2(t))-f(t,\bar{x}(t),\bar{y}(t),\bar{z}(t),\bar{u}_1(t), \bar{u}_2(t)))\rangle dt
\end{array}
$$
and we have used  the fact that$$
y(T)-\bar{y}(T)=\xi-\xi=0, x(0)-\bar{x}(0)=a-a=0.$$

On the other hand, in view of the definition of Hamilton function $H_1$ (see
\eqref{eq:2.2}), the integration $I_1$ can be rewritten as
\begin{equation}
\begin{array}{ll}\label{eq:2.13}
I_1=&E\displaystyle\int_0^T\bigg[l_1(t,x(t),y(t),z(t),u_1(t), \bar u_2(t))\\&~~~~-l_1(t,\bar{x}(t),\bar{y}(t),\bar{z}(t),\bar{u}_1(t), \bar u_2(t))\bigg]dt\\
~~~=&E\displaystyle\int_0^T\bigg[H_1(t,x(t),y(t),z(t),u_1(t),\bar u_2(t), \bar{p}^1(t),
\bar{q}^1(t), \bar{k}^1(t))
\\&-H_1(t,\bar{x}(t),\bar{y}(t),\bar{z}(t),\bar{u}_1(t),
\bar{u}_2(t),\bar{p}^1(t),
\bar{q}^1(t), \bar{k}^1(t))\bigg]dt
\\
&-E\displaystyle\int_0^T\langle\bar{p}^1(t),b(t,x(t),y(t),z(t),u_1(t),\bar u_2(t))-b(t,\bar{x}(t),\bar{y}(t),\bar{z}(t),\bar u_1(t),\bar u_2(t))
\rangle dt\\
&-E\displaystyle\int_0^T\langle\bar{q}^1(t),\sigma(t,x(t),y(t),z(t),u_1(t), \bar u_2(t))
-\sigma(t,\bar{x}(t),\bar{y}(t),\bar{z}(t),\bar u_1(t),\bar u_2(t))\rangle dt
\\&-E\displaystyle\int_0^T\langle \bar{k}^1(t),-(f(t,x(t),y(t),z(t),u_1(t), \bar u_2(t))-f(t,\bar{x}(t),\bar{y}(t),\bar{z}(t),\bar{u}_1(t), \bar{u}_2(t)))\rangle dt\\
~~~=&J_3-J_2,
\end{array}
\end{equation}
where
\begin{eqnarray}\label{eq:2.14}
  \begin{split}
    J_3=&E\displaystyle\int_0^T\bigg[H_1(t,x(t),y(t),z(t),u_1(t),\bar u_2(t), \bar{p}^1(t),
\bar{q}^1(t), \bar{k}^1(t))
\\&-H_1(t,\bar{x}(t),\bar{y}(t),\bar{z}(t),\bar{u}_1(t),
\bar{u}_2(t),\bar{p}^1(t),
\bar{q}^1(t), \bar{k}^1(t))\bigg]dt
  \end{split}
\end{eqnarray}
From the optimality condition \eqref{eq:2.4}, we have
\begin{equation}\label{eq:2.15}
\begin{array}{ll}
&
\bigg\langle H_{1u_1}(t,\bar{x}(t),\bar{y}(t),\bar{z}(t),\bar{u}_1(t),
\bar u_2(t), \bar{p}^1(t),
\bar{q}^1(t), \bar{k}^1(t)), u_1(t)-\bar{u}_1(t)\bigg\rangle \geq 0, a.s.a.e..
\end{array}
\end{equation}

Using convexity of $H_1(t,x,y,z,u_1,\bar{u}_2(t),\bar{p}^1(t),\bar{q}^1(t),\bar{k}^1(t))$ with
respect to $(x,y,z,u_1)$,  and noting \eqref{eq:2.14} and \eqref{eq:2.15},  we have
\begin{equation}\label{eq:2.16}
\begin{array}{ll}
J_3&\geq
E\displaystyle\int_0^T\langle H_{1x}(t,\bar{x}(t),\bar{y}(t),\bar{z}(t),\bar u_1(t),\bar u_2(t),\bar{p}^1(t),
\bar{q}^1(t),
\bar{k}^1(t)),x(t)-\bar{x}(t)\rangle dt\\
&~~+E\displaystyle\int_0^T\langle H_{1y}(t,\bar{x}(t),\bar{y}(t),\bar{z}(t),\bar u_1(t),\bar u_2(t),\bar{p}^1(t),
\bar{q}^1(t),
\bar{k}^1(t)),y(t)-\bar{y}(t)\rangle dt\\
&~~+E\displaystyle\int_0^T\langle H_{1z}(t,\bar{x}(t),\bar{y}(t),\bar{z}(t),\bar u_1(t),\bar u_2(t),\bar{p}^1(t),
\bar{q}^1(t), \bar{k}^1(t)),z(t)-\bar{z}(t)\rangle dt\\
&=J_1.
\end{array}
\end{equation}
Therefore, it follows from \eqref{eq:2.9}, \eqref{eq:2.12},\eqref{eq:2.13}
and  \eqref{eq:2.16} that
$$
\begin{array}{ll}
J(u_1(\cdot), \bar u_2(\cdot))-J(\bar{u}_1(\cdot), \bar u_2 (\cdot))&=I_1+I_2=(J_3-J_2)+I_2\\
&\geq (J_1-J_2)+(-J_1+J_2)=0.
\end{array}
$$
Since $u_1(\cdot)\in{\cal A}_1$ is arbitrary, we conclude that
\begin{eqnarray}\label{eq:2.17}
  \begin{split}
    J_1(\bar{u}_1(\cdot),
\bar{u}_2(\cdot))=\displaystyle\min_{u_1(\cdot)\in {\cal
A}_1}J_1(u_1(\cdot), \bar{u}_2(\cdot)).
  \end{split}
\end{eqnarray}

(ii) Now we consider another stochastic optimal
control problem.
 The system  is the following
controlled
forward-backward stochastic differential equation
\begin{equation}\label{eq:2.18}
\displaystyle\left\{\begin{array}{lll}
dx(t)&=&b(t,x(t),y(t),z(t),\bar u_1(t),{u}_2(t))dt\\
&&+\sigma(t,x(t),y(t),z(t),\bar u_1(t),{u}_2(t))dB(t)\\\displaystyle
dy(t)&=&-f(t,x(t),y(t),z(t),\bar u_1(t),{u}_2(t))dt\\
&&+z(t)dB(t)\\\displaystyle x(0)&=&a\\\displaystyle y(T)&=&\xi,
\end{array} \right.
\end {equation}

where $u_2(\cdot)$ is  any given admissible control in $\cal A_2.$
The cost function is defined as
\begin{equation}\label{eq:3.19}
\begin{array}{ll}
&J_2(\bar u_1(\cdot),{u}_2(\cdot))\\
=&E\bigg[\displaystyle\int_0^Tl_2(t,x(t),y(t),z(t),\bar u_1(t),
{u}_2(t))dt\\
&+\phi_2(x(T))+h_2(y_0)\bigg], \end{array}
\end{equation}
where $(x(\cdot), y(\cdot),z(\cdot))$ is
the solution  to the system \eqref{eq:2.18} corresponding
to the control $u_2(\cdot)\in \cal A_2.$

 The  optimal control problem is minimize
$J(\bar u_1(\cdot), {u}_2(\cdot))$ over $u_2(\cdot) \in {\cal A }_2$.
As in (i), we can  similarly  show  the admissible control $\bar{u}_2(\cdot)$ is  an optimal
control of the problem, i.e,
 \begin{equation}\label{eq:2.20}
  J_1(\bar{u}_1(\cdot),
\bar{u}_2(\cdot))=\displaystyle\min_{u_2(\cdot)\in {\cal
A}_2}J_1(\bar u_1(\cdot), {u}_2(\cdot)).
\end{equation}
So from \eqref{eq:2.17} and \eqref{eq:2.20}, we can conclude
that
  $(\bar u_1(\cdot), \bar u_2(\cdot))$ is
   an
equilibrium point of Problem \ref{pro:2.1}.
The proof is complete.

\end{proof}

\section{Necessary optimality conditions}

\begin{thm}
  Under Assumptions \ref{ass:2.1}-\ref{ass:2.2}, let $(\bar{u}_1(\cdot), \bar{u}_2(\cdot))$
  be a Nash equilibrium point of  Problem \ref{pro:2.1}.
Suppose that  $(\bar{x}(\cdot), \bar{y}(\cdot), \bar{z}(\cdot))$ is
 the state process of the system \eqref{eq:1.1} corresponding to
 the admissible control $(\bar{u}_1(\cdot), \bar{u}_2(\cdot)).$
 Let $({\bar p^i}(\cdot), \bar q^i(\cdot), \bar k^i(\cdot) )$$(i=1,2)$ be  the unique
 solution of the  adjoint equation \eqref{eq:2.1} corresponding $
 (\bar{u}_1(\cdot), \bar{u}_2(\cdot);  \bar{x}(\cdot), \bar{y}(\cdot), \bar{z}(\cdot))$.
 Then we have

   \begin{equation}\label{eq:3.1}
\begin{array}{ll}
\big\langle H_{1u_1}(t,\bar{x}(t),\bar{y}(t),\bar{z}(t),\bar{u}_1(t),
\bar u_2(t), \bar{p}^1(t),
\bar{q}^1(t), \bar{k}^1(t)), u_1-\bar{u}_1(t)\big\rangle \geq 0, \forall
 u_1 \in U_1 a.s.a.e.,
\end{array}
\end{equation}
 \begin{equation}\label{eq:3.2}
\begin{array}{ll}
\big\langle H_{1u_2}(t,\bar{x}(t),\bar{y}(t),\bar{z}(t),\bar{u}_1(t),
\bar u_2(t), \bar{p}^1(t),
\bar{q}^2(t), \bar{k}^2(t)), u_2-\bar{u}_2(t)\big\rangle \geq 0, \forall u_2
\in U_2,  a.s.a.e..
\end{array}
\end{equation}
\end{thm}
\begin{proof}
  Since  $(\bar{u}_1(\cdot), \bar{u}_2(\cdot))$
  be an equilibrium point,
   then \begin{equation}\label{eq:3.3}
    J_1(\bar{u}_1(\cdot),
\bar{u}_2(\cdot))=\displaystyle\min_{u_1(\cdot)\in {\cal
A}_1}J_1(u_1(\cdot), \bar{u}_2(\cdot)).
\end{equation}
and

\begin{equation}\label{eq:3.4}
 J_2(\bar{u}_1(\cdot),
\bar{u}_2(\cdot))=\displaystyle\min_{u_2(\cdot)\in {\cal
A}_2}J_2(\bar u_1(\cdot), {u}_2(\cdot)).
\end{equation}
By \eqref{eq:3.3},  $\bar u_1(\cdot)$  can be regarded as
an optimal control  of the
optimal control problem where
the controlled system is \eqref{eq:2.5} and
the cost functional is
\eqref{eq:3.7}. For this case, it is easy to
see that
the Hamilton function is  $H_1$ (see \eqref{eq:2.2})
and the correspond adjoint equation is $\eqref{eq:2.1}$ for
$i=1,$  and $(\bar x(\cdot), \bar y(\cdot), \bar z(\cdot))$
is the corresponding  optimal state process.  Thus applying the
stochastic maximum principle for the optimal
control of the  forward-backward
stochastic system (see Theorem 3.3 in \cite{Wu9}), we can obtain \eqref{eq:3.1}.
Similarly, from \eqref{eq:3.4}, we can obtain \eqref{eq:3.2}.
The proof is complete.

\section{Conclution}
In this paper, we have discussed two-person non-zero sum differential game
governed by a fully coupled forward-backward stochastic system with the
control process $u(\cdot)$ appearing in the forward diffusion term.
The verification theory is obtained as a sufficient condition for the existence
of open-loop Nash equilibrium point. On the other hand, applying the
stochastic maximum principle for the optimal control
problem of the forward-backward stochastic system,
we derive the the stochastic maximum principle
in a local formulation as a necessary condition for the
existence of open-loop Nash equilibrium point.

\end{proof}

\end{document}